\documentclass[12pt,reqno]{amsart}
\usepackage{graphicx}
\usepackage{amsmath, amssymb}
\usepackage{upgreek}
\usepackage{marvosym}
\usepackage{amsthm}
\usepackage{caption}
\usepackage{subcaption}
\usepackage{amstext}
\usepackage{array}
\usepackage{euler}
\usepackage{times}
\usepackage{lipsum}
\usepackage{kpfonts}
\usepackage[margin=1in]{geometry}
\newtheorem{thm}{Theorem}[section]

\newtheorem{prop}[thm]{Proposition}
\theoremstyle{definition}

\theoremstyle{remark}
\newtheorem{rem}[thm]{Remark}
\numberwithin{equation}{section}
\def\ox{\bar{x}}

\def\oy{\bar{y}}

\def\ph{\varphi}
\def\al{\alpha}
\def\gph{\mbox{\rm gph}\,}
\def\cl{\mbox{\rm cl}\,}
\def\co{\mbox{\rm co}\,}
\def\epi{\mbox{\rm epi}\,}
\def\ra{\rangle}
\def\la{\langle}
\def\emp{\emptyset}
\def\tto{\;{\lower 1pt\hbox{$\rightarrow$}}\kern-10pt
\hbox{\raise 2pt\hbox{$\rightarrow$}}\;}
\def\disp{\displaystyle}
\def\dom{\mbox{\rm dom}\,}

\usepackage{hyperref}
\hypersetup{colorlinks=true,
linkcolor=blue,
filecolor=magenta,
urlcolor=cyan,}
\begin{document}

\title[Nonsmooth optimistic and pessimistic bilevel programs]{Two-level value function approach to nonsmooth optimistic and pessimistic bilevel programs}%
\author{Stephan Dempe}%
\address{S. Dempe: Institut f\"ur Numerische Mathematik und Optimierung, TU Freiberg, Germany}%
\email{dempe@math.tu-freiberg.de}%

\author{Boris S. Mordukhovich}
\address{B.S. Mordukhovich: Department of Mathematics, Wayne State University Detroit, USA}
\email{boris@math.wayne.edu}

\author{Alain B. Zemkoho}
\address{A.B. Zemkoho: School of Mathematics, University of Southampton\\
Southampton, UK}
\email{a.b.zemkoho@soton.ac.uk}
\thanks{The research of the first author has been supported by Deutsche Forschungsgemeinschaft, Project DE 650/10-1. The research of the second author was partially supported by the USA National Science Foundation under grant DMS-1512846, by the USA Air Force Office of Scientific Research under grant No.\,15RT0462, and by the RUDN University Program 5-100.The research of the third author was partially supported by the EPSRC grant {EP/P022553/1}}
\subjclass{49K35, 49K40, 90C26, 90C30}%
\keywords{optimistic and pessimistic bilevel programming, two-level value functions, variational analysis, generalized differentiation, optimality conditions}%

\dedicatory{The paper is dedicated to Ji\v r\'i Outrata in honor of his 70th birthday}%

\date{\today}
\begin{abstract} The authors' paper in \emph{Optimization 63 (2014), 505--533}, see Ref. \cite{DempeMordukhovichZemkohoPessimistic}, was the first one to provide detailed optimality conditions for pessimistic bilevel optimization. The results there were based on the concept of the {\em two-level optimal value function} introduced and analyzed in \emph{SIAM J. Optim. 22 (2012), 1309--1343}; see Ref. \cite{DempeMordukhovichZemkohoTwo-level}, for the case of optimistic bilevel programs. One of the basic assumptions in both of these papers is that the functions involved in the problems are at least continuously differentiable. Motivated by the fact that many real-world applications of optimization involve functions that are nondifferentiable at some points of their domain, the main goal of the current paper is extending the two-level value function approach to deriving new necessary optimality conditions for both  optimistic and pessimistic versions in bilevel programming with nonsmooth data.
\end{abstract}
\maketitle

\tableofcontents

\section{Introduction}

The original framework of {\em bilevel programming}, or {\em bilevel optimization}, goes back to von Stackelberg \cite{Stackelberg1934} and can be written in the form
\begin{equation}\label{Bilevel0la}
\underset{x\in X}\min~F(x,y)\;\mbox{ subject to }\;y\in S(x):=\arg\underset{y\in K(x)}\min~f(x,y),
\end{equation}
where the leader is only in control of the upper level variable $x$, while the follower controls the lower level variable $y$. Here $F\colon\mathbb{R}^n\times\mathbb{R}^m\to\mathbb{R}$ and $f\colon\mathbb{R}^n\times\mathbb{R}^m\to\mathbb{R}$ are called the {\em upper-level} and {\em lower-level} objective functions, respectively, $X$ denotes the upper-level feasible set, while the set-valued mapping $K\colon\mathbb{R}^n\rightrightarrows\mathbb{R}^m$ describes the lower-level feasible set. Note nonetheless that the main focus in the literature (as discussed, i.e., in  \cite{DempeFoundations} and \cite{DempeMordukhovichZemkohoTwo-level}), has been not on the original version \eqref{Bilevel0la} but on its following conventional simplification:
\begin{equation}\label{Standard Bilevel0la}
\underset{x,y}\min~F(x,y)\;\mbox{ subject to }\;x\in X,\;\,y\in S(x),
\end{equation}
were the leader appears to be in control of both variables. The latter formulation is {\em globally} equivalent to the following {\em optimistic} bilevel optimization problem, which is denoted by ($P_o$) and is written in the form
\begin{equation}\label{optimistic 1}
\underset{x\in X}\min~\varphi_{\circ}(x):=\underset{y}\min\big\{F(x,y)|\;y\in S(x)\big\},
\end{equation}
which indicates a certain cooperation between the upper and lower level players. However, \eqref{Standard Bilevel0la} and \eqref{optimistic 1} are {\em not} equivalent {\em locally} as demonstrated and discussed in \cite{DempeMordukhovichZemkohoTwo-level}. To this end, observe also that the leader simply may not be  able to cooperate with the follower in practice. To limit damages that may result from negative choices of the follower, the leader needs to solve the following {\em pessimistic} bilevel program denoted by ($P_p$):
\begin{equation}\label{pesimi 1}
\underset{x\in X}\min~\varphi_{p}(x):=\underset{y}\max\big\{F(x,y)\big|\,y\in S(x)\big\},
\end{equation}
where the {\em min} operation on the right-hand side of \eqref{optimistic 1} is replaced by the {\em max} one; both minimum and maximum are achieved in \eqref{pesimi 1} and in what follows under the assumptions made. It has been well recognized that problems ($P_o$) and ($P_p$) are essentially different from each other; see, e.g., \cite{DempeFoundations} for more discussions. In fact, problem ($P_p$) belongs the most challenging class in bilevel optimization. This is the reason why very little has been known in the literature for solving this problem. There are few publications (see, e.g., \cite{CervinkaMatonohaOutrata2013,WiesemannTsoukalasKleniatiRustem2013}) that propose global optimization techniques to solve ($P_p$). Taking into account serious limitations in global optimization, it is important to develop local optimization methods to solve pessimistic bilevel models. Pursuing this goal, the authors developed in \cite{DempeMordukhovichZemkohoPessimistic}, for the first time in the literature, detailed necessary conditions applied to local optimal solutions of pessimistic bilevel programs. Our approach there is based on the study and application of subdifferential properties of the {\em two-level value function}
\begin{eqnarray}\label{2level}
\varphi_{\circ}(x):=\underset{y}\min\big\{F(x,y)|\;y\in S(x)\big\},\quad x\in\mathbb{R}^n,
\end{eqnarray}
defined in \cite{DempeMordukhovichZemkohoTwo-level} and its symmetric {\em maximizing} version considered in \cite{DempeMordukhovichZemkohoPessimistic} in the form
\begin{eqnarray}\label{2level-pes}
\varphi_p(x):=\underset{y}\max\big\{F(x,y)|\;y\in S(x)\big\},\quad x\in\mathbb{R}^n.
\end{eqnarray}
We refer the reader to \cite{DempeMordukhovichZemkohoTwo-level,DempeMordukhovichZemkohoPessimistic} for more details and the reasons behind the names of \eqref{2level} and \eqref{2level-pes}. As shown in \cite{DempeMordukhovichZemkohoTwo-level}, the subdifferential and related Lipschitz continuity properties of the function $\varphi_{\circ}$ are instrumental for deriving optimality and stationarity conditions for optimistic bilevel programs while the form \eqref{2level-pes} is crucial to obtain similar results for the pessimistic model with smooth data.

Note that the two-level value functions \eqref{2level} and \eqref{2level-pes} are clearly different from the {\em lower-level optimal value function} defined by
\begin{equation}\label{OptimalValueFunction}
\varphi(x):=\underset{y}\min~\big\{f(x,y)\big|\;y\in K(x)\big\}
\end{equation}
that is widely spread in variational analysis under the name of ``marginal function." Ji\v r\'i Outrata was the first \cite{o90} to use \eqref{OptimalValueFunction} in the framework of optimistic bilevel programming via the so-called {\em value function approach} introduced by him and then largely developed in the literature; see, e.g., \cite{DempeFoundations,DempeDuttaMordukhovichNewNece,mor18,MordukhovichNamPhanVarAnalMargBlP,YeZhuOptCondForBilevel1995} and the references therein. Observe also that both functions \eqref{2level} and \eqref{OptimalValueFunction} are {\em intrinsically nonsmooth} even when the initial data of problems \eqref{optimistic 1} and \eqref{pesimi 1} are infinitely differentiable.

Developing here the two-level value function approach to bilevel optimization, recall that in both papers \cite{DempeMordukhovichZemkohoTwo-level,DempeMordukhovichZemkohoPessimistic} we assumed that the functions involved in the descriptions of optimistic and pessimistic problems are at least {\em continuously differentiable}. As discussed, e.g., in \cite{DempeZemkohoOnTheKKTRef}, many real-world applications of bilevel optimization are described by nondifferentiable functions. In order to widen the scope of applications of the results in \cite{DempeMordukhovichZemkohoTwo-level,DempeMordukhovichZemkohoPessimistic}, we extend them in this paper to functions that are {\em not necessarily differentiable}. Before proceeding with optimality conditions for ($P_o$) and ($P_p$) in the nonsmooth settings of Section~\ref{Necessary optimality conditions} and Section~\ref{pessimistic}, respectively, we derive the corresponding sensitivity results for the two-level value functions $\varphi_{\circ}$ in Section~\ref{section 3} and its symmetric maximizing version $\ph_p$ in Section~\ref{pessimistic}. The concluding Section~\ref{remarks} contains some discussions on further research.

The tools of variational analysis and generalized differentiation needed in what follows are summarized in the next section. Throughout the paper we use the standard notation and terminology in the area; see, e.g., \cite{MordukhovichBook2006,RockafellarWetsBook1998}.

\section{Preliminaries from Variational Analysis}\label{section 2}

In this section we recall and briefly overview for the reader's convenience some standard tools of generalized differentiation, which are broadly used in the subsequent variational analysis of both optimistic and pessimistic bilevel programs. More details can be found in the books \cite{MordukhovichBook2006,RockafellarWetsBook1998} and their bibliographies.

Starting with sets, consider $\Theta\subset\mathbb{R}^n$ locally closed around $\ox\in\Theta$ and define the (basic, limiting, Mordukhovich) {\em normal cone} to $\Theta$ at $\ox$ by
\begin{eqnarray}\label{basic normal cone}
\begin{array}{ll}
N_{\Theta}(\ox):=\Big\{v\in\mathbb{R}^n\Big|&\exists\,x_k\to\ox,\;w_k\in\Pi_{\Theta}(x_k),\;\alpha\ge 0\\
&\mbox{such that }\;\al(x_k-w_k)\to v\;\mbox{ as }\;k\to\infty\Big\},
\end{array}
\end{eqnarray}
where $\Pi_\Theta(x)$ stands for the {\em Euclidean projector} of $x$ onto $\Theta$. There are various equivalent descriptions of \eqref{basic normal cone}, which are not used in this paper. Note that the construction \eqref{basic normal cone} is {\em robust} (i.e., the set-valued mapping $x\mapsto N_{\Theta}(x)$ has closed graph), while $N_{\Theta}(\ox)$ is often {\em nonconvex} for nonconvex sets $\Theta$ as, e.g., for $\Theta:=\gph|x|\subset\mathbb{R}^2$, the graph of the simplest nonsmooth convex function, at $\ox=0$. Nevertheless, the normal cone \eqref{basic normal cone} and associated with it subdifferential and coderivative notions for extended-real-valued functions and set-valued mappings/multifunctions enjoy {\em full calculus} based on the {\em variational} and {\em extremal principles} of variational analysis.

The convex closure of the normal cone \eqref{basic normal cone} denoted by
\begin{eqnarray}\label{cnc}
\Bar N_{\Theta}(\ox):=\cl\co N_{\Theta}(\ox),\quad\ox\in\Theta,
\end{eqnarray}
is known as the (Clarke) {\em convexified normal cone} to $\Theta$ at $\ox$. Note that the convexification in \eqref{cnc} may dramatically enlarge the sets of normals (e.g.,$\Bar N_{\Theta}(0)=\mathbb{R}^2$ for $\Theta=\gph|x|$), worsen calculus rules, and lead to the lost of robustness in non-Lipschitzian setting. On the other hand, it has computational advantages over \eqref{basic normal cone} and allows to use for its study and applications powerful tools of convex analysis.

Given an extended-real-valued lower semicontinuous (l.s.c.) function $\psi\colon(-\infty,\infty]$ finite at $\ox$, we define its {\em subdifferential} at this point geometrically via the normal cone \eqref{basic normal cone} to the epigraph $\epi\psi$ by
\begin{equation}\label{Basic Subdifferential}
\partial\psi(\bar x):=\left\{\xi\in\mathbb{R}^n|\;(\xi,-1)\in N_{\text{epi}\psi}\big(\bar x,\psi(\bar x)\big)\right\}
\end{equation}
while noting that \eqref{Basic Subdifferential} admits several equivalent analytical representations. When $\psi$ is locally Lipschitzian around $\bar x$, we get from \eqref{Basic Subdifferential} the (Clarke) {\em generalized gradient}
\begin{equation}\label{Clarke Subdifferential}
\Bar{\partial}\psi(\bar{x}):=\co\partial\psi(\bar{x}),
\end{equation}
which can also be defined by different equivalent ways. Both constructions \eqref{Basic Subdifferential} and \eqref{Clarke Subdifferential} reduce to the classical gradient and subdifferential of convex analysis when $\psi$ is continuous differentiable around $\ox$ and convex, respectively. Observe the convex hull symmetry property for locally Lipschitz functions
\begin{equation}\label{convex hull property 1}
\co\partial\big(-\psi)(\bar{x}\big)=-\co\partial\psi(\bar{x})
\end{equation}
that is largely used in what follows.

Next we recall some notions to (single-valued and set-valued) mappings. Given a closed-graph multifunction $\Psi\colon\mathbb{R}^n\rightrightarrows\mathbb{R}^m$, its {\em coderivative} at the graph point $(\bar{x},\bar{y})\in\gph\Psi:=\{(x,y)\in\mathbb{R}^n\times\mathbb{R}^m|\;y\in\Psi(x)\}$ is defined as a set-valued mapping generated by the normal cone \eqref{basic normal cone} to the graph of $\Psi$ by
\begin{equation}\label{cod-definition}
D^*\Psi(\bar{x},\bar{y})(y^*):=\left\{x^*\in\mathbb{R}^n|\;(x^*,-y^*)\in N_{\text{gph}\,\Psi}(\bar{x},\bar{y})\right\},\;y^*\in\mathbb{R}^m,
\end{equation}
where $\oy=\Psi(\ox)$ is dropped in notation if $\Psi\colon\mathbb{R}^n\to\mathbb{R}^m$ is single-valued. If in the latter case $\Psi$ is continuously differentiable at $\ox$, the coderivative \eqref{cod-definition} is single-valued as well and reduced to the adjoint/transposed Jacobian linearly applied to $y^*$:
\begin{eqnarray*}
D^*\Psi(\ox)(y^*)=\big\{\nabla\Psi(\ox)^*y^*\big\}\;\mbox{ for all }\;y^*\in\mathbb{R}^m.
\end{eqnarray*}
In general the coderivative \eqref{cod-definition} is a positively homogeneous set-valued mapping satisfying comprehensive calculus rules. Observe the scalarization formula
\begin{eqnarray*}
D^*\Psi(\ox)(y^*)=\partial\big\la y^*,\Psi\big)(\ox),\quad y^*\in\mathbb{R}^m,
\end{eqnarray*}
relating the coderivative and subdifferential of single-valued Lipschitzian mappings.

Proceeding further with set-valued mappings, recall that $\Psi\colon\mathbb{R}^n\rightrightarrows\mathbb{R}^m$ is {\em inner semicompact} at some point $\bar{x}$ with $\Psi(\bar{x})\ne\emp$ if for every sequence $x_k\rightarrow\bar{x}$ with $\Psi(x_k)\ne\emp$ there is a sequence of $y_k\in\Psi(x_k)$ that contains a convergent subsequence as $k\rightarrow\infty$. It follows that the inner semicompactness holds in the finite-dimensional setting under consideration whenever $\Psi$ is uniformly bounded around $\bar{x}$, i.e., there exists a neighborhood $U$ of $\bar{x}$ and a bounded set $\Theta\subset\mathbb{R}^m$ such that $\Psi(x)\subset\Theta$ for all $x\in U$.

The following property is more restrictive while bring us to more precise results of the coderivative calculus and its applications. The mapping $\Psi\colon\mathbb{R}^n\rightrightarrows\mathbb{R}^m$ is said to be {\em inner semicontinuous} at $(\bar{x},\bar{y})\in\gph\Psi$ if for every sequence $x_k\rightarrow\bar{x}$ there is a sequence of $y_k\in\Psi(x_k)$ that converges to $\bar{y}$ as $k\rightarrow\infty$. For single-valued mappings $\Psi\colon\mathbb{R}^n\to\mathbb{R}^m$ this property obviously reduces to the continuity of $\Psi$ at $\ox$. Furthermore, if $\Psi\colon\mathbb{R}^n\tto\mathbb{R}^m$ is inner semicompact at $\bar{x}$ with $\Psi(\bar{x})=\{\bar{y}\}$, then $\Psi$ is inner semicontinuous at $(\ox,\oy)$.

Finally in this section, we review some Lipschitzian properties of set-valued mappings used in what follows. The mapping $\Psi\colon\mathbb{R}^n\rightrightarrows\mathbb{R}^m$ satisfies the {\em Lipschitz-like} (Aubin or pseudo-Lipschitz) property around $(\ox,\oy)\in\gph\Psi$ if there are neighborhoods $U$ of $\ox$ and $V$ of $\oy$ together with a constant $\ell\ge 0$ such that
\begin{eqnarray}\label{lip}
\Psi(x)\cap V\subset\Psi(u)+\ell\|x-u\|\mathbb{B}\;\mbox{ for all }\;x,u\in U,
\end{eqnarray}
where $\mathbb{B}$ stands for the closed unit ball of the space in question. This property is a graphical localization of the classical (Hausdorff) local Lipschitz continuity of $\Psi$ around $\ox$, which corresponds to the case of $V=\mathbb{R}^m$ in \eqref{lip}. A complete characterization of the Lipschitz-like property \eqref{lip}, with precise computing the infimum of Lipschitzian moduli $\ell$ in \eqref{lip}, is given by the {\em coderivative/Mordukhovich criterion} in the form
\begin{eqnarray}\label{cc}
D^*\Psi(\ox,\oy)(0)=\{0\}.
\end{eqnarray}
We refer the reader to \cite{mor93,MordukhovichBook2006,mor18,RockafellarWetsBook1998} for more details and numerous applications of \eqref{cc} in variational and nonlinear analysis, optimization, and related areas.

The {\em calmness} property of $\Psi$ at $(\ox,\oy)$ is defined via \eqref{lip} with the fixed vector $u=\ox$ therein. In the case where $V=\mathbb{R}^m$, this property goes back to Robinson \cite{Robinson1981} who called it the ``upper Lipschitz property" of $\Psi$ at $\ox$. It is proved in \cite{Robinson1981} that the upper Lipschitz (and hence calmness) property holds at every point if the graph of $\Psi$ is {\em piecewise polyhedral}, i.e., expressible as the union of finitely many polyhedral sets. Efficient conditions for the validity of the calmness property and its broad applications to variational analysis and optimization were strongly developed by Ji\v r\'i Outrata and his collaborators; see, e.g., \cite{HenrionJouraniOutrataCalmness2002,hen-out01,mor-out07} among many other publications.

\section{Sensitivity Analysis of the Two-Level Value Function}\label{section 3}

This section is devoted to a local sensitivity/stability analysis of the two-level value function $\varphi_{\circ}$ from \eqref{2level}. Besides deriving efficient pointbased upper estimate of the basic subdifferential \eqref{Basic Subdifferential}, we establish here verifiable conditions for the local {\em Lipschitz continuity} of $\varphi_\circ$ around the reference point. From now on unless otherwise stated, the structure of the lower-level feasible solution map is considered to be
\begin{equation}\label{K(x)}
K(x):=\big\{y\in\mathbb{R}^m\big|\;(x,y)\in\Theta_2\cap\vartheta^{-1}_2(\Lambda_2)\big\},
\end{equation}
where the sets $\Theta_2\subset\mathbb{R}^n\times\mathbb{R}^m$ and $\Lambda_2\subset\mathbb{R}^p$ are closed, and where the mapping $\vartheta_2\colon\mathbb{R}^n\rightarrow\mathbb{R}^p$ is Lipschitz continuous.

As major constraint qualifications, we impose the {\em calmness} property at the reference points of the following set-valued mappings:
\begin{equation}\label{X-K-S}
\begin{array}{rll}
\Phi^K(u)&:=&\big\{(x,y)\in\Theta_2\big|\;\vartheta_2(x,y)+u\in\Lambda_2\big\},\\
\Phi^S(u)&:=&\big\{(x,y)\in\gph K\big|\;f(x,y)-\varphi(x)+u\le 0\big\},
\end{array}
\end{equation}
where $\varphi$ is the lower-level optimal value function defined in \eqref{OptimalValueFunction}.

The aforementioned result by Robinson \cite{Robinson1981} ensures the calmness of $\Phi^K$) (resp.\ $\Phi^S$) at any point of the graph provided that the sets $\Theta_2$ and $\Lambda_2$ (resp.\ $\text{gph}\,K$) are polyhedral and the functions $\vartheta_2$ (resp.\ $f$ and $\varphi$) are linear with respect to their variables. Note that $\varphi$  is piecewise linear if, e.g., the set-valued mapping $K$ is defined by a system of linear equalities and/or inequalities while the cost function $f$ is also linear.

Some results obtained below to verify the local Lipschitz continuity of, respectively, the lower-level and two-level value functions in the bilevel programs under consideration impose the following pointbased constraint qualifications:
\begin{equation}\label{coderivative criterion for K}
\Big[(x^*,0)\in\partial\langle u,\vartheta_2\rangle(\bar{x},\bar{y})+N_{\Theta_2}(\bar{x},\bar{y})\;\mbox{ with }\;u\in N_{\Lambda_2}(\vartheta_2(\bar{x},\bar{y}))\Big]\Longrightarrow x^*=0,
\end{equation}
\begin{equation}\label{coderivative criterion for S}
\left.\begin{array}{r}
(x^*,0)\in r\partial(f-\varphi)(\bar{x},\bar{y})+\partial\langle u,\vartheta_2\rangle(\bar{x},\bar{y})+N_{\Theta_2}(\bar{x},\bar{y})\\
r\ge 0,\;u\in N_{\Lambda_2}\big(\vartheta_2(\bar{x},\bar{y})\big)
\end{array}\right\}\Longrightarrow x^*=0.
\end{equation}
These conditions are related to implementing the coderivative criterion \eqref{cc} for the set-valued mappings $K$ and $S$, respectively. It is worth mentioning that standard constraint qualifications as, e.g., the Mangasarian-Fromovitz one, fail to hold for the constraint structure $y\in S(x)$ with $S$ taken from \eqref{Bilevel0la} and $K$ described by simple smooth functions via inequality constraints in the lower-level optimal value function reformulation of bilevel programs; see \cite{DempeZemkohoGenMFCQ} for more discussions. However, condition \eqref{coderivative criterion for S} can be satisfied to provide therefore a valuable constraint qualification in bilevel programming. To illustrate it, we recall the following typical statement in this direction taken from \cite{DempeMordukhovichZemkohoTwo-level}.

\begin{prop}[\bf validity of the coderivative-based constraint qualification]\label{codCQ}
{\em Consider $\varphi$ in \eqref{OptimalValueFunction} with $K(x):=\{y\in\mathbb{R}^m|\;g(y)\le 0\}$, where the functions $f$ and $g$  are convex and smooth. Take $(\bar{x},\bar{y})$  such that $\varphi(\bar{x})<\infty$ and the mapping $M(u):=\{y\in\mathbb{R}^m|\;g(y)+u\le 0\}$ is calm at $(0,\bar{y})$. Then the constraint qualification \eqref{coderivative criterion for S} holds at $(\bar{x},\bar{y})$.}
\end{prop}

The next theorem is the first result in the literature that provides a local sensitivity analysis of the two-level value function \eqref{2level} in bilevel programs described by Lipschitzian functions. In its formulation we use a mild assumption about inner semicompactness of the solution map associated with $\varphi_{\circ}$ and defined by
\begin{equation}\label{S-o}
S_{\circ}(x):=\arg\underset{y}\min\big\{F(x,y)\big|\;y\in S(x)\big\}=\big\{y\in S(x)\big|\;F(x,y)\le\varphi_{\circ}(x)\big\}.
\end{equation}

\begin{thm}[\bf subdifferential estimate and Lipschitz continuity of the two-level value function under inner semicompactness]\label{inner semicompactness} {\em Let in the framework of \eqref{2level} the functions $f$, $\vartheta_2$, and $F$ be Lipschitz continuous around $(\bar{x},y)$ with $y\in S(\bar{x})$ and $(\bar{x}, y)$ with $y\in S_{\circ}(\bar{x})$, respectively. Assume that the sets $\Theta_2$ and $\Lambda_2$ are closed and the set-valued mapping $\Phi^K$ $($resp.\ $\Phi^S)$ is calm at $(0,\bar{x},y)$ with $y\in S(\bar{x})$ (resp.\ at $(0,\bar{x},y)$ with $y\in S_{\circ}(\bar{x})$). Suppose also that $S_{\circ}$ is inner semicompact at $\bar{x}$ and that the constraint qualifications \eqref{coderivative criterion for K} and \eqref{coderivative criterion for S} are satisfied at $(\bar{x},y)$ with $y\in S(\bar{x})$ and $(\bar{x},y)$ with $y\in S_{\circ}(\bar{x})$, respectively. Then the two-level value function $\varphi_{\circ}$ is Lipschitz continuous around $\bar{x}$ and its basic subdifferential \eqref{Basic Subdifferential} at $\ox$ satisfies the following upper estimate:
\begin{equation*}\label{subdif of phi-o detailed}
\begin{array}{l}
\partial\varphi_{\circ}(\bar{x})\subset\underset{y\in S_{\circ}(\bar{x})}\bigcup\;\,\underset{u\in N_{\Lambda_2}(\vartheta_2(\bar{x},y))}\bigcup\;\,\underset{r\ge 0}\bigcup \Big\{x^*\big|\;\disp\sum^{n+1}_{s=1}v_s=1,\\
\qquad\qquad v_s\ge 0,\;u_s\in N_{\Lambda_2}(\vartheta_2(\bar{x},y_s)),\,y_s\in S(\bar{x}),\;s=1,\ldots,n+1,\\
\qquad\qquad(x^{*}_s,0)\in\partial f(\bar{x},y_s)+\partial\langle u_s,\vartheta_2\rangle(\bar{x},y_s)+N_{\Theta_2}(\bar{x},y_s),\;s=1,\ldots,n+1,\\
\qquad\qquad\big(x^*+r\disp\sum^{n+1}_{s=1}v_s x^{*}_s,0\big)\in\partial F(\bar{x},y)+r\partial f(\bar{x},y)+\partial\langle u,\vartheta_2\rangle(\bar{x},y)+N_{\Theta_2}(\bar{x},y)\Big\}.
\end{array}
\end{equation*}}
\end{thm}
\begin{proof}
Pick $\bar{x}^*\in\partial\varphi_{\circ}(\bar{x})$ and deduce from \cite[Theorem~7(ii)]{MordukhovichNamYenSubgradients2009}, by taking into account the inner semicompactness of $S_{\circ}$ at $\bar{x}$ and the Lipschitz continuity of $F$ around $(\bar{x},y)$ with $y\in S_{\circ}(\bar{x})$, the existence of $y\in S_{\circ}(\bar{x})$ with $\bar{x}^*\in x^*+D^*S(\bar{x},y)(y^*)$ for some $(x^*,y^*)\in\partial F(\bar{x},y)$. Noting then that $\gph K=\Phi^K(0)$ and applying \cite[Theorem~4.1]{HenrionJouraniOutrataCalmness2002}, we get that for any $(x^*,y^*)\in N_{\text{gph}\,K}(\bar{x},y)$ there exists $u\in N_{\Lambda_2}(\vartheta_2(\bar{x},y))$ such that $(x^*,y^*)\in\partial\langle u,\vartheta_2\rangle (\bar{x}, y)+N_{\Theta_2}(\bar{x},y)$ with $y\in S(\bar{x})$. This uses the assumed closedness of $\Theta_2$ and $\Lambda_2$, Lipschitz continuity of $\vartheta_2$, and calmness of $\Phi^K$. It follows from this estimate and the coderivative definition \eqref{cod-definition} that for $ x^*\in D^*K(\bar{x},y)(y^*)$ there exist $u\in N_{\Lambda_2}(\vartheta_2(\bar{x}, y))$ satisfying
\begin{eqnarray*}
(x^*,-y^*)\in\partial\langle u,\vartheta_2\rangle(\bar{x},y)+N_{\Theta_2}(\bar{x},y).
\end{eqnarray*}
A clear consequence of this implication tells us that the constraint qualification \eqref{coderivative criterion for K} is a sufficient condition for the coderivative criterion \eqref{cc} to hold for the set-valued mapping $K$ at $(\bar{x},y)$ with $y\in S(\bar{x})$. Combining the latter fact with the inner semicompactness of $S$ (obtained from that of $S_{\circ}$ since $S_{\circ}(x)\subseteq S(x)$ for all $x\in X$), we derive from \cite[Theorem 5.2 (ii)]{MordukhovichNamVariational2005} the Lipschitz continuity of the lower-level value function $\varphi$ around $\bar{x}$.

Repeating now the above process in the case of $S$ and observing that $\gph S=\Phi^S(0)$, we get that for any $(x^*,y^*)\in N_{\text{gph}\,S}(\bar{x},y)$ with $y\in S_{\circ}(\bar{x})$ there is $r\ge 0$ such that
\begin{eqnarray*}
(x^*,y^*)\in r\partial(f-\varphi)(\bar{x},y)+N_{\text{gph}\,K}(\bar{x},y)
\end{eqnarray*}
under the imposed calmness assumption on $\Phi^S$ at $(\bar{x},y)$ and the established local Lipschitz continuity of the lower-level value function $\ph$ around $\ox$. Combining the latter expression with the one obtained above in the case of $K$, it follows from the coderivative definition that for any $x^*\in D^*S(\bar{x},y)(y^*)$ there are $r\ge 0$ and $u\in N_{\Lambda_2}(\vartheta_2(\bar{x},y))$ such that
\begin{eqnarray}\label{inc1}
(x^*,-y^*)\in r\partial(f-\varphi)(\bar{x},y)+\partial\langle u,\vartheta_2\rangle(\bar{x},y)+N_{\Theta_2}(\bar{x},y).
\end{eqnarray}
This allows us to conclude under the assumptions of the theorem that the imposed constraint qualification condition \eqref{coderivative criterion for S} at $(\bar{x},y)$ implies the validity of the coderivative criterion \eqref{cc} for $S$ at the same point and therefore the Lipschitz-like property of $S$ around it. Invoking again \cite[Theorem~5.2(ii)]{MordukhovichNamVariational2005} tells us that the two-level value function $\varphi_{\circ}$ is locally Lipschitz continuous around $\bar{x}$.

The presence of the lower-level value function $\varphi$ in \eqref{inc1} gives us room for a further description of the upper bound for $D^*S(\bar{x},y)(y^*)$ in terms of problem data. Since $\ph$ is a marginal function \cite{MordukhovichBook2006}, it follows from \cite[Theorem~3.38]{MordukhovichBook2006}, the structure of $S$ with $K$ in \eqref{K(x)}, and the normal cone calculus rules in \cite[Theorems~3.4 and 3.8]{MordukhovichBook2006} that for any $x^*\in\partial\varphi(\bar{x})$ we can find $y\in S(\bar{x})$ and $u\in N_{\Lambda_2}(\vartheta_2(\bar{x},y))$ satisfying the inclusion
\begin{eqnarray*}
(x^*,0)\in\partial f(\bar{x},y)+\partial\langle u,\vartheta_2\rangle(\bar{x},y)+N_{\Theta_2}(\bar{x},y).
\end{eqnarray*}
Furthermore, the verified local Lipschitz continuity of $\varphi$ around $\bar{x}$ ensures that
\begin{eqnarray*}
\partial(-\varphi)(\bar{x})\subset\co\partial(-\varphi)(\bar{x})=-co\partial\varphi(\bar{x}),
\end{eqnarray*}
where the equality holds due to the convex hull symmetry property \eqref{convex hull property 1}. Applying now the classical Carath\'eodory's convex hull theorem in $\mathbb{R}^n$ leads us to the following upper estimate of the basic subdifferential of $-\varphi$: for any $x^*\in\partial(-\varphi)(\bar{x})$ there are $v_s\ge 0$, $u_s\in N_{\Lambda_2}(\vartheta_2(\bar{x},y_s))$, $y_s\in S(\bar{x})$ as $s=1,\ldots,n+1$, and
\begin{eqnarray*}
(x^{*}_s,0)\in\partial f(\bar{x},y_s)+\partial\langle u_s,\vartheta_2\rangle(\bar{x},y_s)+N_{\Theta_2}(\bar{x},y_s)\;\mbox{ with }\;\sum^{n+1}_{s=1}v_s=1
\end{eqnarray*}
such that $x^*=-\sum^{n+1}_{s=1}v_s x^{*}_s$. Employing this together with the subdifferential sum rule \cite[Theorem~2.33]{MordukhovichBook2006} in inclusion \eqref{inc1}, we get that for any $x^*\in D^*S(\bar{x},y)(y^*)$ there are $v_s\ge 0$, $u_s\in N_{\Lambda_2}(\vartheta_2(\bar{x},y_s))$, $y_s\in S(\bar{x})$, and
\begin{eqnarray*}
(x^*_s,0)\in\partial f(\bar{x},y_s)+\partial\langle u_s,\vartheta_2\rangle(\bar{x},y_s)+N_{\Theta_2}(\bar{x},y_s)
\end{eqnarray*}
with $s=1,\ldots,n+1$ and $\sum^{n+1}_{s=1}v_s=1$ such that
\begin{eqnarray*}
\Big(x^*+r\sum^{n+1}_{s=1}v_sx^{*}_s,-y^*\Big)\in r\partial f(\bar{x},y)+\partial\langle u,\vartheta_2\rangle(\bar{x},y)+N_{\Theta_2}(\bar{x},y).
\end{eqnarray*}
This verifies the claimed estimate for $\partial\ph_{\circ}$ and completes the proof of the theorem.
\end{proof}

It is definitely appealing to simplify the subdifferential upper estimate for the two-level value function $\ph_{\circ}$ in Theorem~\ref{inner semicompactness} by avoiding the convex hull therein. There are three approaches in the bilevel programming literature to proceed in this direction for the lower-level value function $\ph$ in optimistic programs with smooth data. One way is to assume the {\em full convexity} of the lower-level functions as in \cite{DempeDuttaMordukhovichNewNece,dmn10,mor18,ZemkohoMinmax2011}. The second one is imposing the {\em inner semicontinuity} property of the solution map instead of its inner semicompactness (cf.\ \cite{DempeDuttaMordukhovichNewNece,mor18,MordukhovichNamVariational2005}), and the third approach exploits a {\em difference rule} for regular/Fr\'echet subgradients as in \cite{mor18,MordukhovichNamYenFrechetSubgradients2006}. In what follows we extend the first two approaches to the case of the two-level value function $\ph_\circ$.

The next theorem deals with the fully convex case in the lower-level problem. To simplify the presentation, we use the following multipliers sets:
\begin{eqnarray}\label{Lambda (x,y)}
\qquad \Lambda(\bar{x},y):=\Big\{\gamma\Big|\;0\in\partial_y f(\bar{x},y)+\disp\sum^p_{i=1}\gamma_i\partial_y g_i(\bar{x},y),\;\gamma\ge 0,\;\la\gamma,g(\bar{x},y)\ra=0\Big\},
\end{eqnarray}
\begin{eqnarray}\label{Lambda1}
\begin{array}{ll}
\Lambda^{\circ}(\bar{x},y)&:=\Big\{(r,\beta)\Big|\;r\ge 0,\;\beta\ge 0,\;\la\beta,g(\bar{x},y)\ra=0,\\
&\qquad\qquad\qquad 0\in\partial_y F(\bar{x},y)+r\partial_y f(\bar{x},y)+\disp\sum^p_{i=1}\beta_i\partial_y g_i(\bar{x},y)\Big\}.
\end{array}
\end{eqnarray}

\begin{thm}[\bf subdifferential estimate and Lipschitz continuity of the two-level value function under full convexity]\label{convex case} {\em Let the bilevel program \eqref{Bilevel0la}, \eqref{K(x)} be fully convex, i.e., the functions $f,\,\vartheta_2:=g$, and $F$ are convex in $(x,y)$ with $\Theta_2:=\mathbb{R}^n\times\mathbb{R}^m$ and $\Lambda_2:=\mathbb{R}^p_-$. Assume that the set-valued mapping $\Phi^K$ $($resp.\ $\Phi^S)$ is calm at $(0,\bar{x},y)$ with $y\in S(\bar{x})$ $($resp.\ $(0,\bar{x},y)$ with $y\in S_{\circ}(\bar{x})$), that $S_{\circ}$ is inner semicompact $\bar{x}\in\dom\varphi$, and that the constraint qualification \eqref{coderivative criterion for S} is satisfied at  $(\bar{x},y)$ with $y\in S_{\circ}(\bar{x})$. Then the two-level value function $\varphi_{\circ}$ is Lipschitz continuous around $\bar{x}$ with the subdifferential estimate
\begin{equation*}
\begin{array}{rl}
\partial\varphi_{\circ}(\bar{x})\subset&\underset{y\in S_{\circ}(\bar{x})}\bigcup\,\,\underset{(r,\beta)\in\Lambda^\circ(\bar{x},y)}\bigcup\,\,\underset{\gamma\in\Lambda(\bar{x},y)}\bigcup\Big\{\partial_x F(\bar{x},y)+r\big(\partial_x f(\bar{x},y)-\partial_x f(\bar{x},y)\big)\\
&\qquad\qquad\qquad\qquad\qquad\qquad+\;\disp\sum^p_{i=1}\beta_i\partial_x g_i(\bar{x},y)-r\disp\sum^p_{i=1}\gamma_i\partial_x g_i(\bar{x},y)\Big\}.
\end{array}
\end{equation*}}
\end{thm}
\begin{proof}
It largely takes the same pattern as that of Theorem~\ref{inner semicompactness} with the following essential changes. We can easily check that the lower-level value function $\ph$ is convex under the assumptions made. Given $x^*\in\partial\varphi(\bar{x})$, this allows us--by taking into account the calmness of the set-valued mapping $\Phi^K$ at $(0,\bar{x},y)$ and basic convex analysis as in the proof of \cite[Theorem~4.1]{ZemkohoMinmax2011}--to find a vector $\gamma\in\Lambda(\bar{x},{y})$ satisfying
\begin{eqnarray*}
x^*\in\partial_x f(\bar{x},{y})+\disp\sum^p_{i=1}\gamma_i\partial_x g_i(\bar{x},{y})\;\mbox{ whenever }\;y\in S(\bar{x}).
\end{eqnarray*}
The second observation is that, by the convexity of $\varphi$, the symmetry property \eqref{convex hull property 1} reads
\begin{eqnarray*}
\partial(-\varphi)(\bar{x})\subset-\partial\varphi(\bar{x}).
\end{eqnarray*}
Combining it with the previous facts and using the decomposition property
\begin{eqnarray*}
\partial\psi(x,y)\subset\partial_x\psi(x,y)\times\partial_y\psi(x,y)
\end{eqnarray*}
valid for the convex functions $F$, $f$, and $g$, we complete the proof.
\end{proof}

The following statement also requires the convexity of the bilevel program data while drops some other assumptions of the previous one in a particular setting.

\begin{prop}[\bf sensitivity of the two-level value function for fully convex programs with parameter-independent lower-level solution sets]\label{simple bilevel value function}{\em Consider the bilevel program \eqref{optimistic 1} with $S(x):=S:=\arg\min\{f(y)|\;g(y)\le 0\}$ as $x\in X$, where the functions $f$, $g$, and $F$ convex in $y$ and $(x,y)$, respectively. Then the two-level value function $\ph_\circ$ is locally Lipschitzian around $\bar{x}\in\dom\varphi_{\circ}$, and we have
\begin{eqnarray*}
\partial\varphi_{\circ}(\bar{x})\subset\partial_x F(\bar{x},y)\;\mbox{ for any }\;y\in S_{\circ}(\bar{x}).
\end{eqnarray*}}
\end{prop}
\begin{proof}
The assumptions made ensure the local Lipschitz continuity of the convex function $\varphi_{\circ}$ around $\ox$. Arguing as in the proof of \cite[Theorem~4.1]{ZemkohoMinmax2011} with taking into account Proposition~\ref{codCQ}, pick $\bar{y}\in S_{\circ}(\bar{x})$ and observe that for any $u\in\partial\varphi_{\circ}(\bar{x})$ the pair $(\bar{x},\bar{y})$ is an optimal solution to the nonlinear program
\begin{eqnarray}\label{nlp}
\underset{x,y}\min~\big\{F(x,y)-\langle u,x\rangle\big\}\;\mbox{ subject to }\;(x,y)\in\mathbb{R}^n\times S.
\end{eqnarray}
Applying now \cite[Proposition~5.3]{MordukhovichBook2006} gives us the necessary optimality condition in problem \eqref{nlp} written in the form
\begin{eqnarray*}
(u,0)\in\partial F(\bar{x},\bar{y})+N_{\mathbb{R}^n\times S}(\bar{x},\bar{y}).
\end{eqnarray*}
The decomposition property for the fully convex function $F$ yields the inclusions
\begin{eqnarray*}
u\in\partial_x F(\bar{x},\bar{y})\;\mbox{ and }\;0\in N_S(\bar{y}),
\end{eqnarray*}
which complete the proof, since the latter one is obvious.
\end{proof}

The last result of this section concerns again the general two-level value function \eqref{2level} and provides, besides its local Lipschitz property, a tighter upper estimate of the basic subdifferential \eqref{Basic Subdifferential} with the replacement of the inner semicompactness of the solution map $S_\circ$ in \eqref{S-o} by its inner semicontinuity. Furthermore, we replace the calmness of the set-valued mapping $K$ from \eqref{X-K-S} by the following generalized form of the Mangasarian-Fromovitz constraint qualification expressed via the generalized gradient \eqref{Clarke Subdifferential}:
\begin{equation}\label{CQ3}
\Big[0\in\sum^p_{i=1}\gamma_i\Bar{\partial}g_i(\bar{x},\bar{y}),\;\gamma_i\ge 0,\gamma_i g_i(\bar{x},\bar{y})=0,\;i=1,\ldots,p\Big]\Longrightarrow\gamma_i=0,\;i=1,\ldots,p.
\end{equation}

\begin{thm}[\bf subdifferential estimate and Lipschitz continuity of the two-level value function under inner semicontinuity]\label{inner semicontinuity}{\em Consider the bilevel program \eqref{optimistic 1}, where the functions $f,\,\vartheta_2:=g$, and $F$ are Lipschitz continuous near $(\bar{x},\bar{y})$, where $\Theta_2:=\mathbb{R}^n\times \mathbb{R}^m$ and $\Lambda_2:=\mathbb{R}^p_-$, and where the set-valued mapping $\Phi^S$ is calm at $(0, \bar{x},\bar{y})$. Suppose also that $S_{\circ}$ is inner semicontinuous at $(\bar{x},\bar{y})$, and that the constraint qualifications \eqref{coderivative criterion for S} and \eqref{CQ3} are satisfied at $(\bar{x},\bar{y})$. Then $\varphi_{\circ}$ is Lipschitz continuous around $\bar{x}$ and
\begin{equation*}
\begin{array}{l}
\partial\varphi_{\circ}(\bar{x})\subset\underset{r\ge 0}\bigcup\Big\{x^*\Big|\;(x^*+rx^*_{\varphi},0)\in\partial F(\bar{x},\bar{y})+r\partial f(\bar{x},\;\bar{y})+\disp\sum^{p}_{i=1}\beta_i\partial g_i(\bar{x},\bar{y}),\\
(x^*_{\varphi},0)\in\Bar{\partial}f(\bar{x},\bar{y})+\disp\sum^{p}_{i=1}\gamma_i\Bar{\partial}g_i(\bar{x},\bar{y}),\, \beta\ge 0, \la\beta, g\ra(\bar{x},\bar{y})=0,\, \gamma\ge 0, \la\gamma,g\ra(\bar{x},\bar{y})=0\Big\}.
\end{array}
\end{equation*}}
\end{thm}
\begin{proof}
By $S_{\circ}(x)\subset S(x)$ for all $x\in X$ we have that $S$ is also inner semicontinuous at $(\bar{x},\bar{y})$. In addition to constraint qualification \eqref{CQ3}, it follows from \cite[Theorem~5.9]{MordukhovichNamPhanVarAnalMargBlP} that for any $u\in\Bar{\partial}\varphi(\bar{x})$ there is a vector $\gamma$ with $\gamma\ge 0$ (componentwise) and $\la\gamma,g(\bar{x},\bar{y})=0$ such that
\begin{eqnarray*}
(u,0)\in\Bar{\partial}f(\bar{x},\bar{y})+\sum^p_{i=1}\gamma_i\Bar{\partial}g_i(\bar{x},\bar{y}).
\end{eqnarray*}
As discussed above, the coderivative qualification condition \eqref{coderivative criterion for K} ensures that the lower-level value function $\varphi$ is Lipschitz continuous around $\bar{x}$. Then it follows from \cite[Theorem~7(i)]{MordukhovichNamYenSubgradients2009}, by the Lipschitz continuity of $F$ around $(\bar{x},\bar{y})$ and the inner semicontinuity of $S_{\circ}$ at the same point, that for any $\bar{x}^*\in\partial\varphi_{\circ}(\bar{x})$ there is a pair $(x^*,y^*)\in\partial F(\bar{x},\bar{y})$ with $\bar{x}^*\in x^*+D^*S(\bar{x},\bar{y})(y^*)$. Remembering the Lipschitz continuity of $\varphi$ around $\bar{x}$, we get as in the proof of Theorem~\ref{inner semicompactness} that for any $x^*\in D^*S(\bar{x},\bar{y})(y^*)$ there are $r\ge 0$, $\gamma$, and $\beta$ with
\begin{eqnarray*}
\gamma\ge 0,\;\la\gamma,g(\bar{x},\bar{y})\ra=0,\;\beta\ge 0,\;\mbox{ and }\;\la\beta,g(\bar{x},\bar{y})\ra=0
\end{eqnarray*}
satisfying the following inclusions, by taking into account the convex hull property \eqref{convex hull property 1}:
\begin{eqnarray*}
(x^*_{\varphi}, 0)\in\Bar{\partial}f(\bar{x},\bar{y})+\disp\sum^{p}_{i=1}\gamma_i\bar{\partial} g_i(\bar{x},y),
\end{eqnarray*}
\begin{eqnarray*}
(x^*+rx^*_{\varphi},-y^*)\in r\partial f(\bar{x},y)+\disp\sum^{p}_{i=1}\beta_i\partial g_i(\bar{x},y).
\end{eqnarray*}
Combining these facts, we arrive at the upper estimate of $\partial\varphi_{\circ}(\bar{x})$ claimed in the theorem. The Lipschitz continuity of the two-level value function $\varphi_{\circ}$ is obtained as in the proof of Theorem~\ref{inner semicompactness}; cf.\ also the proof of \cite[Theorem~5.2(i)]{MordukhovichNamVariational2005}.
\end{proof}

\begin{rem}[\bf discussions on sensitivity analysis of the two-level value function]\label{rem2level} {\rm The following observations are in order to make:

{\bf(i)} Assuming in the framework of Theorem~\ref{inner semicompactness} that $S_{\circ}(\bar{x})=S(\bar{x})=\{y\}$ and that the multiplier $u$ in the upper estimate  of $\partial \varphi(\bar{x})$ is unique, we see that all the upper bounds of $\partial\varphi_{\circ}(\bar{x})$ obtained in  Theorems~\ref{inner semicompactness}, \ref{convex case}, and \ref{inner semicontinuity} coincide provided that the functions involved are subdifferentially regular (i.e., $\partial\psi(\ox)=\Bar\partial\psi(\ox)$ in the setting under consideration) and the sets $\Theta_2$ and $\Lambda_2$ are adjusted accordingly.

{\bf(ii)} If all the functions involved in \eqref{2level} are continuously differentiable around the reference points with $\Theta_2:=\mathbb{R}^n\times\mathbb{R}^m$ and $\Lambda_2:=\mathbb{R}^p_-$, then the upper estimates that we get for $\partial\varphi_{\circ}$ in Theorems~\ref{inner semicompactness} and \ref{inner semicontinuity} reduce to those obtained in \cite[Theorem~5.9]{DempeMordukhovichZemkohoTwo-level}. This will induce the same observation concerning the necessary optimality conditions for the optimistic bilevel program ($P_o$) obtained in Theorem~\ref{optimality cond for Po}.

{\bf(iii)} Note that for the upper estimates of the basic subdifferential of $\varphi_{\circ}$ derived in Theorems~\ref{inner semicompactness}, \ref{convex case}, and \ref{inner semicontinuity},  the coderivative constraint qualification \eqref{coderivative criterion for S} is in fact not necessary. The latter condition comes into play only to obtain the Lipschitz continuity of the aforementioned two-level value function.}
\end{rem}

\section{Necessary Optimality Conditions for Optimistic Bilevel Programs}\label{Necessary optimality conditions}

Our focus in this section is on deriving necessary optimality conditions for the original optimistic formulation of the bilevel programming problem ($P_o$) by using the two-level value function \eqref{2level} and the sensitivity results obtained for it in Section~\ref{section 3}. In addition to the expression of the lower-level feasible set-valued mapping $K$ in \eqref{K(x)}, we consider now the representation of the upper-level feasible set given by
\begin{equation}\label{X}
X:=\Theta_1\cap\vartheta^{-1}_1(\Lambda_1),
\end{equation}
where the sets $\Theta_1\subset\mathbb{R}^n$ and $\Lambda_1\subset\mathbb{R}^k$ are always assumed to be closed, and where the mapping $\vartheta_1\colon\mathbb{R}^n\times\mathbb{R}^m\rightarrow\mathbb{R}^k$ is locally Lipschitz continuous. Similarly to the lower-level setting \eqref{X-K-S} the {\em calmness} assumption is imposed on the set-valued mapping
\begin{eqnarray}\label{PhiX}
\Phi^X(u):= \left\{x\in\Theta_1\big|\;\vartheta_1(x)+u\in\Lambda_1\right\}.
\end{eqnarray}

The following necessary optimality conditions for the optimistic bilevel program $(P_o)$ from \eqref{optimistic 1} are based on sensitivity analysis of the two-level value function $\ph_\circ$ and implement all the independent results for such an analysis obtained in Theorems~\ref{inner semicompactness}, \ref{convex case}, and \ref{inner semicontinuity} while combining them in the three statements of the next theorem.

\begin{thm}[\bf necessary optimality conditions for the original version in optimistic bilevel programming]\label{optimality cond for Po}{\em
Let $\bar{x}$ be a local optimal solution to (P$_o$), where $\vartheta_1$ is Lipschitz continuous around this point, and where the set-valued mapping $\Phi^X$ from \eqref{PhiX} is calm at $(0,\bar{x})$. Then the following assertions hold:
\begin{itemize}
\item[\bf(i)] Under the assumptions of Theorem~{\rm\ref{inner semicompactness}} there exist $y\in
S_{\circ}(\bar{x})$, $(\alpha,u,r)\in\mathbb{R}^{k+p+1}$, $(u_s,v_s)\in\mathbb{R}^{p+1}$, $y_s\in S(\bar{x})$, and $x^*_s\in\mathbb{R}^n$ as $s=1,\ldots,n+1$ such that
\begin{eqnarray}
\left(r\disp\sum^{n+1}_{s=1}v_s x^{*}_s,0\right)\in\partial F(\bar{x},y)+r\partial f(\bar{x},y)\qquad\qquad\qquad\qquad\label{op1}\\+\left[\partial\langle\alpha,\vartheta_1\rangle(\bar{x})+N_{\Theta_1}(\bar{x})\right]\times\left\{0\right\}
+\;\partial\langle u,\vartheta_2\rangle(\bar{x},y)+N_{\Theta_2}(\bar{x},y),\nonumber\\
\qquad \mbox{ for }\;s=1,\ldots, n+1,\;(x^{*}_s,0)\in\partial f(\bar{x},y_s)+\partial\langle u_s,\vartheta_2\rangle(\bar{x},y_s)+N_{\Theta_2}(\bar{x},y_s),\label{op2}\\
\mbox{for }\;s=1,\ldots,n+1,\;u_s\in N_{\Lambda_2}\big(\vartheta_2(\bar{x},y_s)\big),\;v_s\ge 0,\;\disp\sum^{n+1}_{s=1}v_s=1,\label{op3}\\
u\in N_{\Lambda_2}\big(\vartheta_2(\bar{x},y)\big),\label{op5}\\
r\ge 0,\;\alpha\in N_{\Lambda_1}\big(\vartheta_1(\bar{x})\big).\label{op6}
\end{eqnarray}
\item[\bf(ii)] Under the assumptions of Theorem~{\rm\ref{convex case}} there exist $y\in S_{\circ}(\bar{x})$, $\alpha$, $\beta$, $\gamma$ and $r$ such
that condition \eqref{op6} holds together with
\begin{eqnarray}
0\in\partial_x F(\bar{x},y)+r\big(\partial_x f(\bar{x},y)-\partial_x f(\bar{x},y)\big)+\disp\sum^p_{i=1}\beta_i\partial_x g_i(\bar{x},y)\qquad\label{conv 1}\\
-\;\;r\sum^p_{i=1}\gamma_i\partial_x g_i(\bar{x},y)+\partial\langle\alpha,\vartheta_1\rangle(\bar{x})+N_{\Theta_1}(\bar{x}),\nonumber\\
0\in\partial_y F(\bar{x},y)+r\partial_y f(\bar{x},y)+\disp\sum^p_{i=1}\beta_i\partial_y g_i(\bar{x},y),\label{conv 2}\\
0\in\partial_y f(\bar{x},y)+\disp\sum^p_{i=1}\gamma_i\partial_y g_i(\bar{x},y),\label{conv 3}\\
\beta\ge 0,\;\la\beta,g(\bar{x},{y})\ra=0,\label{conv 4} \\
\gamma\ge 0,\;\la\gamma,g(\bar{x},y)\ra=0.\label{conv 5}
\end{eqnarray}

\item[\bf(iii)] Under the assumptions of Theorem~{\rm\ref{inner semicontinuity}} there exist $x^*$, $\alpha$, $\beta$, $\gamma$, and $r$ such that we
have conditions \eqref{op6} and \eqref{conv 4}--\eqref{conv 5} (where $y:=\bar{y}$) together with
\begin{eqnarray}
\begin{split}(rx^*,0)\in \partial F(\bar{x},\bar{y})+&r\partial f(\bar{x},\bar{y})+\disp\sum^{p}_{i=1}\beta_i\partial g_i(\bar{x},\bar{y})\\+&\left[\partial\langle\alpha,\vartheta_1\rangle(\bar{x})+N_{\Theta_1}(\bar{x})\right]\times\left\{0\right\},
\end{split}\label{iscn 1}\\
(x^*,0)\in\Bar{\partial}f(\bar{x},\bar{y})+\disp\sum^{p}_{i=1}\gamma_i\Bar{\partial}g_i(\bar{x},\bar{y}).\label{iscn 2}
\end{eqnarray}
\end{itemize}}
\end{thm}
\begin{proof}
Under the assumptions of either (i), or (ii), or (iii) the two-level value function $\varphi_{\circ}$ is Lipschitz continuous near $\bar{x}$ as proved in the aforementioned theorems. Then remembering the closedness of the set $X=\Theta_1\cap\vartheta^{-1}(\Lambda_1)$, we get by applying the necessary optimality condition from \cite[Proposition~5.3]{MordukhovichBook2006} to problem $(P_o)$ due to its structure in \eqref{optimistic 1} with the feasible set \eqref{X} that
\begin{eqnarray}\label{nc-opt}
0\in\partial\varphi_{\circ}(\bar{x})+N_X(\bar{x}).
\end{eqnarray}
The imposed calmness constraint qualification on the set-valued mapping $\Phi^X$ at $(0,\bar{x})$ allows us to deduce from \cite[Theorem~4.1]{HenrionJouraniOutrataCalmness2002} that for any $x^*\in N_{X}(\bar{x})$ there is a vector $\alpha\in N_{\Lambda_1}(\vartheta_1(\bar{x}))$ satisfying the inclusion
\begin{eqnarray*}
x^*\in\partial\langle\alpha\vartheta_1\rangle(\bar{x})+ N_{\Theta_1}(\bar{x}).
\end{eqnarray*}
Combining the latter with \eqref{nc-opt} and the upper estimate for $\partial\varphi_{\circ}(\bar{x})$ from Theorems~\ref{inner semicompactness}, \ref{convex case}, and  \ref{inner semicontinuity}, we arrive at the conclusions in assertions (i)--(iii) of the theorem.
\end{proof}

\begin{rem}[\bf discussion on necessary optimality conditions in optimistic bilevel programming]\label{dis-opt} {\rm The following remarks discuss some important features of the obtained necessary optimality conditions for the original optimistic model $(P_o)$ and their comparison with known results for optimistic bilevel programs.

{\bf(i)} As follows from the proof of Theorem~\ref{optimality cond for Po}, the Lipschitz continuity of the two-level value function $\ph_\circ$ is used to obtain the intermediate necessary optimality condition in form \eqref{nc-opt}. It is worth mentioning that the necessary optimality condition \eqref{nc-opt} holds under the more general constraint qualification
\begin{eqnarray}\label{singular-opt}
\partial^\infty\ph_\circ(\ox)\cap\big(-N_X(\ox)\big)=\{0\},
\end{eqnarray}
where $\partial^\infty\psi(\ox)$ is the {\em singular subdifferential} of $\psi\colon\mathbb{R}^n\to(-\infty,\infty]$ at $\ox\in\dom\psi$ given by
\begin{eqnarray}\label{ss}
\partial^\infty\psi(\ox):=\left\{\xi\in\mathbb{R}^n|\;(\xi,0)\in N_{\text{epi}\psi}\big(\bar x,\psi(\bar x)\big)\right\}.
\end{eqnarray}
It is well known in variational analysis \cite{MordukhovichBook2006,RockafellarWetsBook1998} that the singular subdifferential \eqref{ss} of an extended-real-valued l.s.c.\ function in finite dimensions completely characterized its locally Lipschitzian property around the point in question via $\partial^\infty\psi(\ox)=\{0\}$. On the other hand, the local Lipschitz continuity of $\ph_\circ$ is by far not necessary for the validity of the qualification requirement \eqref{singular-opt}, while the upper subdifferential estimates obtained in Theorems~\ref{inner semicompactness}, \ref{convex case}, and \ref{inner semicontinuity} do not use the Lipschitz continuity of the two-level value function \eqref{2level}. In particular, under the validity of \eqref{singular-opt} the coderivative constraint qualification \eqref{coderivative criterion for S} needed just for the Lipschitz continuity of $\ph_\circ$ can be dropped. It happens, e.g., when $X:=\mathbb{R}^n$ in the upper-level problem.

{\bf(ii)} To compare Theorem~\ref{optimality cond for Po} with previous work on the subject, consider the original optimistic bilevel optimization problem (P$_{\circ}$) in the particular case where $\Theta_1:=\mathbb{R}^n$, $\Lambda_1:=\mathbb{R}^k_-$, $\Theta_2:=\mathbb{R}^n\times\mathbb{R}^m$, and $\Lambda_2:=\mathbb{R}^p_-$. Then the necessary optimality conditions obtained in Theorem~\ref{optimality cond for Po}(i,ii,iii) reduce to those derived in \cite[Theorem~5.1]{DempeDuttaMordukhovichNewNece}, \cite[Theorem~4.1]{DempeDuttaMordukhovichNewNece}, and \cite[Theorem~6.4]{MordukhovichNamPhanVarAnalMargBlP}, respectively, for the now conventional form \eqref{Standard Bilevel0la} in optimistic bilevel programming with the additional statement on  $y\in S_\circ(\ox)$. In the framework of smooth data in optimistic bilevel programs, the results somewhat similar to (while different from) Theorem~\ref{optimality cond for Po}(i) can also be found in \cite{DempeZemkohoGenMFCQ,YeZhuOptCondForBilevel1995}. Observe that the aforementioned papers employ the so-called {\em partial calmness} constraint qualification introduced in Ye and Zhu \cite{YeZhuOptCondForBilevel1995}, which is replaced here by the calmness of the set-valued mapping $\Phi^S$. In fact, these constraint qualifications are largely related; cf.\ the proof of \cite[Theorem~4.10]{DempeZemkohoGenMFCQ}. Note also that the (Mangasarian-Fromovitz type) upper-level and lower-level regularity conditions used in \cite{DempeDuttaMordukhovichNewNece,MordukhovichNamPhanVarAnalMargBlP} yield (being significantly stronger) the fulfilment of the calmness conditions for the set-valued mappings $\Phi^X$ and $\Phi^K$ imposed in Theorem~\ref{optimality cond for Po}. Furthermore, the coderivative constraint qualification \eqref{coderivative criterion for K} readily follows from the lower-level regularity. The reader is referred to \cite{DempeDuttaMordukhovichNewNece,DempeZemkohoGenMFCQ,MordukhovichNamPhanVarAnalMargBlP,YeZhuOptCondForBilevel1995} and the bibliographies therein for the definitions and more detailed discussions on the partial calmness, upper-level and lower-level regularity, and related constraint qualifications. We particularly emphasize remarkable developments by Henrion and Surowiec \cite{hen-sur11} who strongly investigated relationships between calmness, partial calmness as well as novel constraint qualifications introduced and implemented by them to derive new necessary optimality conditions for the conventional optimistic bilevel model with a convex lower-level problem.}
\end{rem}

\section{Sensitivity and Optimality Conditions for Pessimistic Bilevel Programs}\label{pessimistic}

This section is devoted to the study of the pessimistic bilevel program $(P_p)$ formulated in \eqref{pesimi 1} with its initial data $X$, $S$, $\varphi$, and $K$ defined as in \eqref{X}, \eqref{Bilevel0la}, \eqref{OptimalValueFunction}, and \eqref{K(x)}, respectively. Developing the two-level value function approach, we obtain here the results on sensitivity analysis and necessary optimality conditions for $(P_p)$ similar to those established for the original optimistic model $(P_o)$ in Sections~\ref{section 3} and \ref{Necessary optimality conditions}. In fact, a strong advantage of our two-level value function approach to bilevel programming is the possibility to simultaneously apply it to sensitivity and optimality issues in both optimistic and pessimistic versions. In this section we demonstrate it for the case of the pessimistic model \eqref{pesimi 1} by involving the maximizing two-level value function $\ph_p$ from \eqref{2level-pes} instead of $\ph_\circ$ from \eqref{2level} used above. Taking into account the relationship
\begin{equation}\label{optimistic two-level value function}
\ph_p(x)=-\ph^o_p(x)\;\mbox{ with }\;\varphi^o_p(x):=\underset{y}\min\big\{-F(x,y)\big|\;y\in S(x)\big\}
\end{equation}
and that the local Lipschitzian property is invariant with respect to the negative sign, we are able to deduce sensitivity results for the pessimistic program from those obtained in Theorems~\ref{inner semicompactness}, \ref{convex case}, and \ref{inner semicontinuity}. To proceed in this way, define the solution set
\begin{equation}\label{S^p-o}
S^p_o(x):=\left\{y\in S(x)\big|\;F(x,y)+\varphi^o_p(x)\ge 0\right\}.
\end{equation}

\begin{thm}[\bf sensitivity analysis for pessimistic bilevel programs]\label{max 2 level value function} {\em Considering the maximizing two-level value function $\varphi_p$ in \eqref{2level-pes}, the following assertions hold:
\begin{itemize}
\item[\bf(i)] Let the assumptions of Theorem~{\rm\ref{inner semicompactness}} be satisfied for the function $\varphi^o_p$ from
\eqref{optimistic two-level value function}
with the set $S_{\circ}$ \eqref{S-o} replaced by $S^p_o$ \eqref{S^p-o}. Then $\varphi_p$ is Lipschitz continuous around $\bar{x}$ and we have the subdifferential upper estimate
\begin{eqnarray*}
\partial\varphi_p(\bar{x})&\subset&\left\{-\disp\sum^{n+1}_{t=1}\eta_t x^*_t\Big|\;\disp\sum^{n+1}_{s=1}v_s=1,\;\disp\sum^{n+1}_{t=1}\eta_t=1,\right.\\
&&v_s\ge 0,\;u_s\in N_{\Lambda_2}\big(\vartheta_2(\bar{x},y_s)\big),\;y_s\in S(\bar{x}),\;s=1,\ldots,n+1,\\
&&(x^*_s,0)\in\partial f(\bar{x}, y_s)+\partial\langle u_s,\vartheta_2\rangle(\bar{x},y_s)+N_{\Theta_2}(\bar{x},y_s),\;s=1,\ldots,n+1,\\
&&\eta_t\ge 0,\;r_t\ge 0,\;u_t\in N_{\Lambda_2}\big(\vartheta_2(\bar{x},y_t)\big),\;y_t\in S^p_o(\bar{x}),\;t=1,\ldots,n+1,\\
&&(x^*_t+r_t\disp\sum^{n+1}_{s=1}v_s x^*_s, 0)\in \partial (-F)(\bar{x}, y_t)+ r_t \partial f(\bar{x}, y_t)\\
&&\qquad\qquad\qquad+\partial\langle u_t,\vartheta_2\rangle(\bar{x},y_t)+N_{\Theta_2}(\bar{x},y_t),\;t=1,\ldots,n+1\Big\}.
\end{eqnarray*}

\item[\bf(ii)] Let the assumptions of Theorem~{\rm\ref{convex case}} be satisfied for $\varphi^o_p$ with $S_{\circ}$ replaced by $S^p_o$. Then $\varphi_p$
is Lipschitz continuous around $\bar{x}$ and we have the following inclusion, where $\Lambda(\bar{x},y_t)$ and $\Lambda^o(\bar{x},y_t)$ are given in \eqref{Lambda (x,y)} and \eqref{Lambda1}, respectively:
\begin{eqnarray*}
\begin{array}{ll}
\partial \varphi_p(\bar{x})\subset\Big\{-\disp\sum^{n+1}_{t=1}\eta_t x^*_t\Big|\;\disp\sum^{n+1}_{t=1}\eta_t=1,\\
\qquad\qquad\qquad\eta_t\ge 0,\;(r_t,\beta^t)\in\Lambda^o(\bar{x},y_t),\;\gamma\in\Lambda(\bar{x},y_t).\;y_t\in S^p_o(\bar{x}),\;t=1,\ldots,n+1,\\
\qquad\qquad\qquad x^*_t\in\partial_x(-F)(\bar{x},y_t)+r_t\big(\partial_x f(\bar{x}, y_t)-\partial_x f(\bar{x},y_t)\big)\\
\qquad\qquad\qquad\qquad\qquad +\disp\sum^p_{i=1}\beta^t_i\partial_x g_i(\bar{x},y_t)-r_t\disp\sum^p_{i=1}\gamma_i\partial_x g_i(\bar{x},y_t),\;t=1,\ldots,n+1\Big\}.
\end{array}
\end{eqnarray*}

\item[\bf(iii)] Let the assumptions of Theorem~{\rm\ref{inner semicontinuity}} be satisfied for $\varphi^o_p$ with
$S_{\circ}$ replaced by $S^p_o$. Then $\varphi_p$ is Lipschitz continuous near $\bar{x}$ and we have the inclusion
\begin{eqnarray*}
\begin{array}{ll}
\partial\varphi_p(\bar{x})\subset\Big\{-\disp\sum^{n+1}_{t=1}\eta_t x^*_t\Big|\;\disp\sum^{n+1}_{t=1}\eta_t=1,\\
\qquad\qquad\quad\eta_t\ge 0,\;r_t\ge 0,\;\beta^t\ge 0,\;\la\beta^t),g(\bar{x},\bar{y}\ra=0,\;t=1,\ldots,n+1,\\
\qquad\qquad\quad(x^*_\varphi,0) \in\Bar{\partial}f(\bar{x},\bar{y})+\disp\sum^p_{i=1}\gamma_i\Bar{\partial}g_i(\bar{x},\bar{y}),\;\gamma\ge 0,\;\la\gamma,g(\bar{x},\bar{y}\ra=0,\\
\qquad\qquad\quad(x^*_t+r_tx^*_\varphi,0)\in\partial(-F)(\bar{x},\bar{y})+r_t\partial f(\bar{x},\bar{y})+\disp\sum^p_{i=1}\beta^t_i\partial g_i(\bar{x},\bar{y}),\;t=1,\ldots,n+1\Big\}.
\end{array}
\end{eqnarray*}
\end{itemize}}
\end{thm}
\begin{proof}
Due to \eqref{optimistic two-level value function}, the Lipschitz continuity of $\varphi_p$ follows from Theorems~\ref{inner semicompactness}, \ref{inner semicontinuity}, and \ref{convex case}, respectively. To verify the upper estimate of $\partial\varphi_p$ in (i), observe from the proof of Theorem~\ref{inner semicompactness} that for any $x^*\in\partial \varphi^o_p(\bar{x})$ there exist a number $r\ge 0$ and vectors $y\in S^p_o(\bar{x})$, $u\in N_{\Lambda_2}(\vartheta_2(\bar{x},y))$ such that
$(x^*,0)=(x^{*1}, y^{*1})+(x^{*2}, y^{*2})+(x^{*3},y^{*3})$ with the couples
\begin{eqnarray*}
\begin{array}{ll}
&(x^{*1}, y^{*1})\in(-F)(\bar{x},y)+r\partial f(\bar{x},y),\;(x^{*2},y^{*2})\in r\partial(-\varphi)(\bar{x})\times\{0\},\\
&\mbox{and }\;(x^{*3},y^{*3})\in\partial\langle u,\vartheta_2\rangle(\bar{x},y)+N_{\Theta_2}(\bar{x},y).
\end{array}
\end{eqnarray*}
Furthermore, it follows from \eqref{convex hull property 1} that $\partial\varphi_p(\bar{x})\subset-\co\partial\varphi^o_p(\bar{x})$, and that for any $x^*\in\co\partial\varphi^o_p(\bar{x})$ there exist $x^*_t$ as $t=1,\ldots,n+1$ satisfying $x^*=\sum^{n+1}_{t=1}\eta_t x^*_t$, where
\begin{eqnarray*}
\begin{array}{ll}
&\disp\sum^{n+1}_{t=1}\eta_t=1,\;y_t\in S^p_o(\bar{x}),\;\eta_t\ge 0,\;r_t\ge 0,\\
&u_t\in N_{\Lambda_2}\big(\vartheta_2(\bar{x},y_t)\big),\;(x^*_t,0)=(x^{*1}_t,y^{*1}_t)+(x^{*2}_t,y^{*2}_t)+(x^{*3}_t,y^{*3}_t),\\
&(x^{*1}_t,y^{*1}_t)\in(-F)(\bar{x},y_t)+r\partial f(\bar{x},y_t),\;(x^{*2}_t,y^{*2}_t)\in r\partial(-\varphi)(\bar{x})\times\{0\},\\
&\mbox{and }\;(x^{*3}_t,y^{*3}_t)\in\partial\langle u_t,\vartheta_2\rangle(\bar{x},y_t)+N_{\Theta_2}(\bar{x},y_t)\;\mbox{ as }\;t=1,\ldots,n+1.
\end{array}
\end{eqnarray*}
This yields the existence of $x^*_\varphi\in\co\partial\varphi(\bar{x})$ with$ (x^*_t+r_tx^*_\varphi,0)=(x^{*1}_t,y^{*1}_t)+(x^{*3}_t,y^{*3}_t)$,
\begin{eqnarray*}
(x^{*1}_t,y^{*1}_t)\in(-F)(\bar{x},y_t)+r\partial f(\bar{x},y_t),\mbox{ and }
(x^{*3}_t,y^{*3}_t)\in\partial\langle u_t,\vartheta_2\rangle(\bar{x},y_t)+N_{\Theta_2}(\bar{x},y_t).
\end{eqnarray*}
Invoking now the upper estimate of $\partial\varphi(\bar{x})$ from the proof of Theorem~\ref{inner semicompactness} verifies assertion (i). The proofs of assertions (ii) and (iii) are similar with the usage of the results obtained in Theorems~\ref{convex case} and \ref{inner semicontinuity}, respectively.
\end{proof}

Proceeding as in the case of the original optimistic bilevel program \eqref{optimistic 1}, we derive now necessary optimality conditions for nonsmooth pessimistic bilevel programs.

\begin{thm}[\bf necessary optimality conditions in nonsmooth pessimistic programming]\label{optimality cond for Pp} {\em Let $\bar{x}$ be a local optimal solution to problem ($P_p$) with the constraints defined at the beginning of this section, where the function $\vartheta_1$ is locally Lipschitz continuous around $\bar{x}$, and where the sets $\Theta_1$ and $\Lambda_1$ are closed. If the set-valued mapping $\Phi^X$ from \eqref{PhiX} is calm at $(0,\bar{x})$, then the following assertions hold:
\begin{itemize}
\item[\bf(i)] Under the assumptions of Theorem~{\rm\ref{max 2 level value function}}(i) there exist $y_t\in S^p_o(\bar{x})$, $\eta_t$, $r_t$, $u_t$ with $t=1,
\ldots, n+1$, $y_s\in S(\bar{x})$, $v_s$, and $u_s$ with $s=1,\ldots,n+1$ such that conditions \eqref{op2}--\eqref{op3} and \eqref{op6} $($where $r:=r_t)$ are satisfied together with
\begin{eqnarray}
\disp\sum^{n+1}_{t=1}\eta_t x^*_t\in\partial\langle\alpha,\vartheta_1\rangle(\bar{x})+N_{\Theta_1}(\bar{x})\;\mbox{ for all }\;t=1,\ldots,n+1,\label{pes 1}\\
\begin{split}\Big(x^*_t+r_t\disp\sum^{n+1}_{s=1}v_s x^*_s,0\Big)\in\partial(-F)(\bar{x},y_t)+&r_t\partial f(\bar{x},y_t)+\\\partial\langle u_t,\vartheta_2\rangle(\bar{x},y_t)+&N_{\Theta_2}(\bar{x}, y_t),\end{split}\label{pes 2}\\
\mbox{ for }\;t=1,\ldots,n+1,\;\eta_t\ge 0,\;\disp\sum^{n+1}_{t=1}\eta_t=1,\label{pes 3}\\
\mbox{ for }\;t=1,\ldots,n+1,\;u_t\in N_{\Lambda_2}\big(\vartheta_2(\bar{x},y_t)\big).\label{pes 4}
\end{eqnarray}

\item[\bf(ii)] Under the assumptions of Theorem~{\rm\ref{max 2 level value function}}(ii) for any $y\in S(\bar{x})$ there exist $\gamma$,
$y_t\in S^p_o(\bar{x}),\;\eta_t,\;r_t,$ and $\beta^t$ with $t=1,\ldots,n+1$ such that the relationships in \eqref{op6} $($with $r:=r_t)$, \eqref{conv 3} $($with $y:=y_t)$, \eqref{conv 5} $($with $y:= y_t)$,  \eqref{pes 1}, and \eqref{pes 3} hold together with the following conditions satisfying for all $t=1,\ldots,n+1$:
\begin{eqnarray}
x^*_t\in\partial_x(-F)(\bar{x},y_t)+r_t\big(\partial_x f(\bar{x},y_t)-\partial_x f(\bar{x},y)\big)\qquad\qquad\qquad\qquad\nonumber\\
+\disp\sum^p_{i=1}\beta^t_i\partial_x g_i(\bar{x},y_t)-r_t\disp\sum^p_{i=1}\gamma_i\partial_x g_i(\bar{x},y),\\
0\in\partial_y F(\bar{x},y_t)+r_t\partial_y f(\bar{x},y_t)+\disp\sum^p_{i=1}\beta^t_i\partial_y g_i(\bar{x},y_t),\\
\beta^t\ge 0,\;\la\beta^t,g(\bar{x},y_t)\ra=0.\label{pess 6}
\end{eqnarray}

\item[\bf(iii)] Under the assumptions of Theorem~{\rm\ref{max 2 level value function}}(iii) there exist $\eta_t,\;r_t,\;u_t$ as $t=1,\ldots,n+1$
such that the relationships in \eqref{op6} $($with $r:=r_t)$, \eqref{conv 5} $($with $y:=\bar{y})$, \eqref{iscn 2}, \eqref{pes 1}, \eqref{pes 3}, and \eqref{pess 6} are satisfied together with the condition
\begin{eqnarray*}\left(x^*_t+r_tx^*_\varphi,0\right)\in\partial(-F)(\bar{x},\bar{y})+r_t\partial f(\bar{x},\bar{y})+\disp\sum^p_{i=1}\beta^t_i\partial g_i(\bar{x},\bar{y}),\;t=1,\ldots,n+1.
\end{eqnarray*}
\end{itemize}}
\end{thm}
\begin{proof}
It follows the lines in the proof of Theorem~\ref{optimality cond for Po} with replacing $\varphi_{\circ}$ by $\varphi_p$ therein and applying the subdifferential estimates for $\ph_p$ obtained in assertions (i), (ii), and (iii) of Theorem~\ref{max 2 level value function}, respectively.
\end{proof}

\begin{rem}[\bf comparison and further discussions]\label{rem-pes} {\rm Let us mention here some other relations and consequences of the developed approach and obtained results.

{\bf(i)} If all the functions involved in ($P_p$) are smooth and $\Theta_1:=\mathbb{R}^n$, $\Lambda_1:=\mathbb{R}^k_-$, $\Theta_2:=\mathbb{R}^n\times \mathbb{R}^m$, and $\Lambda_2:=\mathbb{R}^p_-$ therein, the necessary optimality conditions for pessimistic bilevel program established in Theorem~\ref{optimality cond for Pp} agree with those obtained in \cite{DempeMordukhovichZemkohoPessimistic}. Note that the latter paper explores also other approaches to pessimistic bilevel programs with smooth data that are related to the generalized equation and Karush-Kuhn-Tucker reformulations. These two approaches lead us to necessary optimality conditions, which are largely different from those obtained here by using the two-level value function approach.

{\bf(ii)} Theorem~\ref{max 2 level value function} sheds new light on the local Lipschitz continuity and upper estimates for the basic subdifferential of the standard maximum functions corresponding to the replacement of the optimal solution map $S$ in \eqref{2level-pes} by the constraint mapping $K$ of type \eqref{K(x)}. Proceeding as in the proof of Theorem~\ref{optimality cond for Pp} with $\Theta_1:=\mathbb{R}^n$, $\Lambda_1:=\mathbb{R}^k_-$, $\Theta_2:=\mathbb{R}^n\times\mathbb{R}^m$, and $\Lambda_2:=\mathbb{R}^p_-$ allows us to recover, in particular, the necessary optimality conditions for minimax problem obtained in \cite{ZemkohoMinmax2011}.}
\end{rem}

\section{Concluding Remarks}\label{remarks}

This paper shows that the two-level value function approach is instrumental to handle both original optimistic and pessimistic models of bilevel programming with nonsmooth data. Developing this approach allows us to carry out a local sensitivity analysis (Lipschitz continuity and subdifferential estimates) of the two-level value function and derive necessary optimality conditions for local solutions to optimistic and pessimistic bilevel programs without any differentiability assumptions.

Our future research aims to implement the obtained stability and optimality results to designing appropriate numerical algorithms for solving bilevel programs, which belong to the class of ill-posed optimization problems; see \cite{zem16} for more discussions and some developments in this direction for optimistic bilevel models with smooth data. Eventually we plan to proceed with developing generalized versions of Newton's method by involving the construction of the second-order subdifferential/generalized Hessian for extended-real-valued functions \cite{MordukhovichBook2006}.

Another goal of our future research is combine advantages of the two-level value function approach with the one initiated in \cite{hen-sur11} and implemented there for conventional optimistic bilevel programs with twice continuously differential initial data and convex lower-level problems. The latter approach produced in \cite{hen-sur11} first-order and second-order necessary optimality conditions under certain calmness constraint qualifications. We plan to extend this approach, by marring it to the two-level value function one under consideration, to the settings of the original optimistic and pessimistic bilevel programs with smooth and nonsmooth data. The aforementioned notion of the second-order subdifferential seems to be instrumental in the absence of twice continuous differentiability.

\bibliographystyle{amsplain}

\end{document}